\documentclass{amsart}
\usepackage{amsfonts}
\newtheorem{lem}{Lemma}

\newcommand\numberthis{\addtocounter{equation}{1}\tag{\theequation}}
\DeclareMathOperator{\Pf}{Pf}
\DeclareMathOperator{\sgn}{sgn}

\newtheorem{thm}{Theorem}
\newtheorem{cor}{Corollary}
\newtheorem{conj}{Conjecture}
 \usepackage{color}
\allowdisplaybreaks
\begin{document}
 \title[Small gaps]{Small gaps of GOE}

\author[Feng, Tian and Wei]{Renjie Feng, Gang Tian and Dongyi Wei}
\address{Beijing International Center for Mathematical Research, Peking University, Beijing, China, 100871.}

\email{renjie@math.pku.edu.cn}
\email{gtian@math.pku.edu.cn}
 \email{jnwdyi@pku.edu.cn}

\date{\today}
   \maketitle
    \begin{abstract}
In this article, we study the smallest gaps of the Gaussian
orthogonal ensemble (GOE) with the joint density \eqref{6}.
The main result is that the smallest gaps, after normalized by $n$, will tend to a Poisson distribution, and the limiting density of the $k$-th normalized smallest gaps is  $ 2{}x^{2k-1}e^{-x^{2}}/(k-1)!$. \end{abstract}

\section{Introduction}
The problem regarding the spacings  of eigenvalues is one of the most important problems in random matrix theory. The gap probability of eigenvalues for the classical random matrices GOE, GUE, GSE and its universality for more general ensembles such as the Wigner matrices  are studied intensively and  pretty well-known \cite{AGZ, BEY, DG, ey,  Me, TV}.  There are also results on the single spacing of eigenvalues for the classical matrices and some universal ensembles \cite{ey, N, T, TV}. 
But there are only a few results regarding the extreme gaps. The motivations to study the extreme gaps of eigenvalues of random matrices come from many different areas such as conjectures regarding the extreme gaps for zeros of Riemann zeta function \cite{D, Mo}, quantum chaos \cite{BBB, BG} and quantum information theory \cite{STKZ}.  
 Now let's give a brief review of the existing results.

The way to derive the smallest gaps for the determinantal point processes is basically well established. The distributions of the smallest gaps of CUE and GUE were first obtained by Vinson  using the moment method \cite{V}. In \cite{So}, Soshinikov investigated the smallest gaps for any determinantal point process on the real line with a translation invariant kernel and proved that some Poisson distribution can be observed in the limit.  Then Ben Arous-Bourgade in \cite{BB}  applied Soshinikov's method to derive the joint density of the  smallest gaps of CUE and GUE, and they proved that the $k$-th smallest gaps of CUE and GUE, normalized by $n^{ {4}/{3}}$, have the limiting density  proportional to \begin{equation}\label{s2}x^{3k-1}e^{-x^{3}},\end{equation}
here, the joint density of GUE is 
 \begin{align}\label{j2}\frac{1}{Z_{n,2}}e^{-n\sum\limits_{i=1}^n\lambda_i^2}\prod_{1\leq i<j\leq n}|\lambda_i-\lambda_j|^2,\end{align}
 where $Z_{n,2}$ is the normalization constant.  Later on, Figalli-Guionnet  derived the smallest gaps for some invariant multimatrix Hermitian ensembles \cite{FG}. As a remark,  the determinantal structure is essential in the proofs in \cite{BB, FG, So, V}.


In \cite{FW},  we derived the smallest gaps of  eigenangles of C$\beta$E  beyond the determinantal  case for any positive integer $\beta$. For the two-dimensional point process 
 $$\chi^{(n)}=\sum_{i=1}^n\delta_{(n^{\frac {\beta+2}{\beta+1}}(\theta_{i+1}-\theta_i), \theta_i)},$$
we proved that $\chi^{(n)}$ tends   
to a Poisson point process $ \chi$ as $n\to\infty$ with intensity\begin{align*}&\mathbb{E}\chi(A\times I)= \frac{A_{\beta}|I|}{2\pi}\int_Au^{\beta}du,\end{align*} 
where $A\subset \mathbb R_+$ is any bounded Borel set, $I\subset (-\pi, \pi)$ is an interval, $|I|$ is the Lebesgue measure
of $I$ and \begin{equation}\label{abb}A_\beta=(2\pi)^{-1}\frac{(\beta/2)^\beta(\Gamma(\beta/2+1))^3}{\Gamma(3\beta/2+1)\Gamma(\beta+1)}.\end{equation}
 In particular,  the result holds for COE, CUE and CSE  with $$A_1=\frac 1{24},\,\,A_2=\frac 1{24\pi},\,\,A_4=\frac 1{270\pi}$$ correspondingly. 
 
 As a direct consequence, let's denote $t^n_{k,\beta}$ as the $k$-th smallest gap of C$\beta$E where $t_{1,\beta}^n<t_{2,\beta}^n<t_{3,\beta}^n\cdots$ and define $$ \tau^n_{k,\beta}=n^{(\beta+2)/(\beta+1)}\times (A_{\beta}/(\beta+1))^{{1}/({\beta+1})}t^n_{k,\beta},$$ then  we have the limiting density
\begin{equation}\label{sc}\lim_{n\to+\infty}\mathbb P(\tau^n_{k,\beta}\in A)= \int_A\frac{\beta+1}{(k-1)!} x^{k(\beta+1)-1}e^{-x^{\beta+1}}dx.\end{equation}
 For   general C$\beta$E,  there is  no determinantal  structure as CUE and the whole proof in \cite{FW} is based on the Selberg integral.

The decay order $\sqrt{32\log n}/n$ of the largest gaps of CUE and GUE was predicted by Vinson  in \cite{V}, and the proof is given by Ben Arous-Bourgade in \cite{BB}. The same decay order for the largest gaps of some invariant multimatrix Hermitian matrices is derived by Figalli-Guionnet in \cite{FG}. Recently, the fluctuations of the largest gaps of CUE and GUE are further derived in \cite{FW1}.

But there is no previous result on the extreme gaps for GOE. There are some essential difficulties for GOE compared with GUE. For GUE, it  is a determinantal point process so that one can express the point correlation functions explicitly and apply the 
Hadamard-Fisher inequality to control the estimates. This is not the case for GOE even though GOE has a Pfaffian structure. One can only express the point correlation functions of GOE as  integrals of the joint density. This causes many difficulties and all the proofs require delicate estimates of the integrals. 
In this paper, we will derive the smallest gaps of GOE and this is the first result regarding the extreme gaps of GOE. Our arguments follow the approach we developed in \cite{FW}.

For GOE, the joint density of the eigenvalues is \begin{align}\label{6}\frac{1}{G_n}e^{-\sum\limits_{i=1}^n\lambda_i^2/2}\prod_{1\leq i<j\leq n}|\lambda_i-\lambda_j|\end{align}with respect to the Lebesgue measure  on $\mathbb{R}^n$. Here,  the normalization constant\begin{equation}{G_{n}}:=\int_{\mathbb{R}^n}d\lambda_1\cdots d\lambda_n
e^{-\sum\limits_{i=1}^n\lambda_i^2/2}\prod_{i<j}|\lambda_i-\lambda_j|
\end{equation}is (Proposition 4.7.1 in  \cite{For})\begin{equation}{G_{n}}=(2\pi)^{n/2}\prod_{j=0}^{n-1}\frac{\Gamma(1+(j+1)/2)}{\Gamma(3/2)}.
\end{equation}
In fact, one may view the above joint density as the one-component log-gas of $n$ particles with charge $q=1$  on the real line  and the  Hamiltonian is
$$H(\lambda_1,\cdots, \lambda_n)=\sum_{i=1}^n\lambda_i^2/2-\sum_{i<j}\log |\lambda_i-\lambda_j|.$$
Now let's consider the following point process on $\mathbb R_+$ \begin{align}\label{chiii}&\chi^{(n)}=\sum_{i=1}^{n-1}\delta_{n(\lambda_{(i+1)}-\lambda_{(i)})},
\end{align}where $\lambda_{(i)}\ (1\leq i\leq n)$ is the increasing rearrangement of $\lambda_{i}\ (1\leq i\leq n)$. The main result of this article  is \begin{thm}\label{thm1}
Let $\lambda_1,\cdots,\lambda_{n}$ be eigenvalues of GOE, 
then the point process $\chi^{(n)} $ will converge to a Poisson point process $ \chi$ as $n\to+\infty$ with intensity\begin{align*}&\mathbb{E}\chi(A)= \frac{1}{4}\int_Audu,
\end{align*}where $ A\subset\mathbb{R}_{+}$ is any bounded Borel set.
\end{thm}
As a direct consequence of Theorem \ref{thm1}, we will have 
\begin{cor}\label{cor}Let's denote $t^n_k$ as the $k$-th smallest gap and $ \tau^n_k=2^{-3/2}nt^n_k$,  then 
\begin{equation}\label{s1}\lim_{n\to+\infty}\mathbb P(\tau^n_k\in A)=\int_A\frac{2}{(k-1)!} x^{2k-1}e^{-x^{2}}dx\end{equation}
for any bounded interval $ A\subset\mathbb{R}_{+}$.
\end{cor}
 
As a remark, the factor $1/ 4$ in Theorem \ref{thm1} is quite meaningful. In fact,  the main observation in Lemma \ref{lem9}  is that 
\begin{equation}\label{dttime}1/ 4=(G_{n-2k,k}/G_{n})^{1 /k},\end{equation} i.e.,  its $k$-th power is the quotient of the generalized partition function of the two-component
log-gas (where the system consists of $n-2k$ particles with charge $q=1$ and $k$ particles with charge $q=2$) and the partition function of the one-component log-gas (see \S \ref{pre} for these definitions). 
Actually, one of the  crucial ideas in the whole proof  is that one can bound  the integrals of the joint density of  one-component log-gas by the generalized partition functions of two-component log-gas (see Lemma \ref{lem12} in \S \ref{twotwo}). 

 

\subsection{Remarks}
One may  consider the smallest gaps for G$\beta$E with the joint density \begin{align}\label{jbeta}\frac{1}{Z_{n,\beta}}e^{-n\beta \sum\limits_{i=1}^n\lambda_i^2/2}\prod_{1\leq i<j\leq n}|\lambda_i-\lambda_j|^\beta \end{align}
 where $\beta>0$  and $Z_{n,\beta}$ is the normalization constant.  Note that compared with \eqref{6},  the joint density \eqref{jbeta} with $\beta=1$ has a factor $n$ in the exponential function,  this  will cause an extra factor $\sqrt n$ for the spacings of eigenvalues, i.e., the smallest gap is of order $n^{-3/2}$ under the joint density $\eqref{jbeta}$ with $\beta=1$ for GOE.

By  comparing the limiting densities \eqref{s2},\eqref{s1} with \eqref{sc} with $\beta=1, 2$, 
it is believed that the smallest gaps of G$\beta$E  have the same limiting behaviors as C$\beta$E and we propose the following conjecture. 
\begin{conj} \label{conjec}
Let's denote $t^n_{k,\beta}$ as the $k$-th smallest gap of G$\beta$E with the joint density \eqref{jbeta}, then there exists  some constant $c_\beta$ depending on $\beta$  such that  \begin{equation}  \tau^n_{k,\beta}=c_{\beta}n^{(\beta+2)/(\beta+1)}t^n_{k,\beta}\end{equation} has  the limiting density 
\begin{equation}\frac{\beta+1}{(k-1)!} x^{k(\beta+1)-1}e^{-x^{\beta+1}}\end{equation}
as $n \to \infty$.\end{conj}
It seems that our strategy to prove the smallest gaps for GOE can be used to prove that of G$\beta$E  and more general ensembles with the joint density 
 \begin{align}\frac{1}{Z_{n,\beta, V}}e^{-n\beta \sum\limits_{i=1}^nV(\lambda_i)}\prod_{1\leq i<j\leq n}|\lambda_i-\lambda_j|^\beta.\end{align}
It's very likely that Conjecture \ref{conjec} is still true for general potential $V(x)$ with mild assumptions instead of $x^2/2$. One of the difficulties is to prove some  identity as \eqref{dttime} or some asymptotic limit as in Lemma 4 in \cite{FW}. 
Actually, there are some results only in the case of $\beta=2$, for example,  Vinson derived the smallest gaps when the potential $V(x)$ is a real analytic potential which is regular and whose equilibrium measure supported on a single interval \cite{V}; while in \cite{FG}, Figalli-Guionnet derived the universal results for the smallest gaps for some invariant multimatrix Hermitian matrices.

Recently, in \cite{B, ds}, Bourgade and  Landon-Lopatto-Marcinek proved the universality for the extreme gaps in the bulk of the general Hermitian and symmetric Wigner matrices with assumptions.

In the end, let's mention some conjectures and results regarding  the local statistics of many other important point processes that are related to the classical random matrix models.  
The local statistics of eigenvalues of the Laplacian of several integrable systems are
believed to follow Poisson statistics \cite{BT},  while for generic chaotic systems,
such as non-arithmetic surfaces of negative curvature, they are expected to follow
the GOE \cite{BG} (see \cite{BBB}  for the results about the smallest gaps between the first $N$ eigenvalues of the Laplacian on a rectangular billiard as $N$ large enough). In number theory, the local statistics of zeros
of Riemann zeta function are expected to follow the GUE \cite{D, Mo}.   
In high energy physics, the numerical results in \cite{GV, GV2} indicate that the local behaviors of the SYK model, which describes $n$ (an even integer) random interacting Majorana modes on a quantum dot \cite{black}, are similar to GOE ($n=0$ mod 8), GUE($n=2,6$ mod $8$) and GSE($n=4$ mod $8$), i.e., the single SYK model encodes the three classical random matrix models. We also refer to \cite{FTD1, FTD2, FTD3} for the mathematical results on the SYK model.

The organization of this article is as follows. In Section 2, we review some basic facts about the joint density of GOE, two-component log-gas, the Hermite polynomials and the Pfaffian of an antisymmetric matrix. In Section 3, we prove an important identity for the generalized partition functions of the two-component log-gas of GOE. Its proof uses certain properties of Pfaffians and Hermite polynomials regarding GOE. In Section 4, we introduce and discuss two more auxiliary point processes. In Section 5, we prove the non-existence of successive small gaps. In Section 6, we establish certain integral inequalities for the two-component log-gas. In Section 7, we complete the proof of Theorem \ref{thm1}.

\bigskip
\textbf{Acknowledgement:} We would like to thank P. Bourgade, O. Zeitouni, G. Ben Arous and P. Forrester for  many helpful discussions.


\section{Preliminaries} \label{pre}

In this section, we will first review some results regarding the joint density of GOE, two-component log-gas  and the Hermite polynomials. Then we will recall the 
definition and several basic properties of the Pfaffian of an antisymmetric matrix.

 As explained in \cite{Me} (see (5.2.9) and (6.1.2)-(6.1.5) in \cite{Me}),  we can rewrite the joint density \eqref{6} as 
 $$|J_n(x_1,\cdots, x_n)|/G_n,$$
 where $J_n(x_1,\cdots, x_n)$ can be expressed in terms of a determinant as  \begin{align}\label{46}&J_n(x_1,\cdots, x_n):=e^{-\sum\limits_{i=1}^nx_i^2/2}\prod_{j<i}(x_i-x_j)=c_n\det[\varphi_{i-1}(x_j)]_{i,j=1,\cdots,n},
\end{align}
 and the partition function of the integration constant is \begin{align}\label{42}{G_{n}}=&\int_{\mathbb{R}^n}dx_1\cdots dx_n
|J_n(x_1,\cdots, x_n)|\\ \nonumber=&n!c_n\int_{x_1<\cdots<x_n}\!dx_1\cdots dx_n\det[\varphi_{i-1}(x_j)]_{i,j=1,\cdots,n}, 
\end{align}
where $c_n>0$ is a constant depending only on $n$ and\begin{equation}\label{vard}\varphi_{j}(x)=(2^jj!\sqrt{\pi})^{-1/2}e^{-x^2/2}H_j(x)=(2^jj!\sqrt{\pi})^{-1/2}e^{x^2/2}
(-d/dx)^je^{-x^2}\end{equation}are the "oscillator wave functions" orthogonal over $\mathbb{R}$ such that  \begin{align}\label{45}&\int_{\mathbb{R}}\varphi_{j}(x)\varphi_{k}(x)dx=\delta_{jk}=\left\{\begin{array}{ll}
                                                                                     1, & \text{if}\ j=k;\\
                                                                                     0, & \text{otherwise.}
                                                                                   \end{array}
\right.\end{align}
 Here, $\{H_j(x)\}$ are Hermite polynomials. 
From the following recurrence relations of Hermite polynomials \begin{align}\label{40}&H_{j+1}(x)=2xH_j(x)-2jH_{j-1}(x),\ H_j'(x)=2jH_{j-1}(x),
\end{align} one deduces\begin{align}\label{41}&\sqrt{2}\varphi_{j}'(x)=\sqrt{j}\varphi_{j-1}(x)-\sqrt{j+1}\varphi_{j+1}(x),\,\,\,j\geq 0,
\end{align}where we denote $\varphi_{-1}(x)=0$. 
Moreover, we have (see (5.47) in \cite{For})\begin{align}\label{54}&H_j(x)=\sum_{m=0}^{[j/2]}(-1)^m2^{j-m}{j\choose 2m}\frac{(2m)!}{2^mm!}x^{j-2m},
\end{align}and $H_n(x)$ is uniquely determined by the first equation of \eqref{40} and the initial condition $H_0(x)=1,\ H_1(x)=2x.$ From the expression of $H_j(x)$, we also have\begin{align}\label{43}&\text{span}\{x^j;j\in\mathbb{Z}\cap[0,n]\}=\text{span}\{H_j(x);j\in\mathbb{Z}\cap[0,n]\},\\
\label{44}&V_n:=\text{span}\{x^je^{-x^2/2};j\in\mathbb{Z}\cap[0,n]\}=\text{span}\{\varphi_j(x);j\in\mathbb{Z}\cap[0,n]\}.
\end{align}

Actually, the joint density \eqref{6} can be identified with the Boltzmann factor of a particular one-component log-gas (see \S 1.4 in \cite{For}). 
One can also define the two-component log-gas for the system that consists of $n_1$ particles with charge $q=1$ and $n_2$ particles
with charge $q = 2$. The two-component log-gas provides an interpolation between GOE ($\beta=1$) and GSE ($\beta=4$) (see \cite{R} and \S 6.7 in \cite{For}). For the two-component log-gas, the generalized partition function of the integration constant is \begin{equation}\label{gaa}{G_{n_1,n_2}}:=\int_{\mathbb{R}^{n_1+n_2}}d\lambda_1\cdots d\lambda_{n_1+n_2}e^{-\sum\limits_{i=1}^{n_1+n_2}q_i\lambda_i^2/2}
\prod_{j<k}|\lambda_j-\lambda_k|^{q_jq_k},
\end{equation}where $q_j=1$ for $1\leq j\leq n_1$ and $q_j=2$ for $n_1+1\leq j\leq n_1+n_2.$ 


Let\begin{align*}&J_{n_1,n_2}(x_1,\cdots, x_{n_1+n_2}):=e^{-\sum\limits_{i=1}^{n_1+n_2}q_ix_i^2/2}\prod_{j<i}(x_i-x_j)^{q_iq_j},
\end{align*}where $q_j=1$ for $1\leq j\leq n_1$ and $q_j=2$ for $n_1+1\leq j\leq n_1+n_2.$ Then we have\begin{align*}&J_{n_1,n_2}(x_1,\cdots, x_{n_1+n_2})=\prod_{j=1}^{n_2}\frac{\partial}{\partial y_{n_1+2j}}J_{n_1+2n_2}(y_1,\cdots, y_{n_1+2n_2}),
\end{align*}where the right hand side is evaluated at $y_j=x_j$ for $j\in \mathbb{Z}\cap[1, n_1]$ and $ y_{n_1+2j}=y_{n_1+2j-1}=x_{n_1+j}$ for $j\in \mathbb{Z}\cap[1, n_2]. $ Therefore, differentiating \eqref{46}, we have\begin{align*}&J_{n_1,n_2}(x_1,\cdots, x_{n_1+n_2})=c_{n_1+2n_2}\det\left[\begin{array}{l}
                                     [\varphi_{i-1}(x_j)]_{i=1,\cdots,n_1+2n_2;\ j=1,\cdots,n_1} \\
                                     \left[\begin{array}{l}
                                             \varphi_{i-1}(x_j) \\
                                             \varphi_{i-1}'(x_j)
                                           \end{array}
                                     \right]_{\text{\tiny\!\!$\begin{array}{l}i=1,\cdots,n_1+2n_2\\ j=n_1+1,\cdots,n_1+n_2\end{array}$}}
                                   \end{array}
\right]
\end{align*}and\begin{align}\label{47}&{G_{n_1,n_2}}=\int_{\mathbb{R}^{n_1+n_2}}dx_1\cdots dx_{n_1+n_2}
|J_{n_1,n_2}(x_1,\cdots, x_{n_1+n_2})|=\\ \nonumber(n_1!)c_{n_1+2n_2}&\int_{\Delta_{n_1}\times\mathbb{R}^{n_2}}\!dx_1\cdots dx_{n_1+n_2}\det\left[\begin{array}{l}
                                     [\varphi_{i-1}(x_j)]_{i=1,\cdots,n_1+2n_2;\ j=1,\cdots,n_1} \\
                                     \left[\begin{array}{l}
                                             \varphi_{i-1}(x_j) \\
                                             \varphi_{i-1}'(x_j)
                                           \end{array}
                                     \right]_{\text{\tiny\!\!$\begin{array}{l}i=1,\cdots,n_1+2n_2\\ j=n_1+1,\cdots,n_1+n_2\end{array}$}}
                                   \end{array}
\right].
\end{align}Here, $ \Delta_j=\{x_1<\cdots<x_j\}\subset\mathbb{R}^{j}$ is a simplex. We also have \begin{equation}\label{ggg}G_n=G_{n,0}.\end{equation}

Now let's recall the definition of the Pfaffian of an antisymmetric matrix of even size (see Definition 6.1.4 in \cite{For}):  Let $X=[\alpha_{ij}]_{i,j=1,\cdots,2N}$ be an antisymmetric matrix. Then the Pfaffian of $X$ is defined by\begin{align}\label{Pf}\Pf X&=\sum_{P(2l)>P(2l-1)}^*\varepsilon(P)\prod_{l=1}^{N}\alpha_{P(2l-1)P(2l)}\\ \nonumber&=\frac{1}{2^NN!}\sum_{P\in S_{2N}}\varepsilon(P)\prod_{l=1}^{N}\alpha_{P(2l-1)P(2l)},
\end{align}where in the first summation the $*$ denotes that the sum is restricted to distinct terms only and $\varepsilon(P)$ is the signature of the permutation $P$.

When $X$ is a  $2N\times2N$ antisymmetric matrix and $B$ is a general $2N\times2N$ matrix, then we have (see (6.12) and (6.35) in \cite{For}) \begin{align}\label{49}(\Pf X)^2=\det X,\ \Pf(B^TXB)=(\det B)(\Pf X),\ \Pf(\lambda X)=\lambda^N\Pf X.
\end{align}Here, the third identity follows from the definition \eqref{Pf}.

\section{Partition functions of two-component log-gas}\label{23e}

In this section, we will prove Lemma \ref{lem9} for the two-component log-gas of GOE. The proof is based on the properties of Pfaffians and Hermite polynomials regarding GOE (see \cite{DG, For, Me, R} for more details) and some integration
techniques from Chapter 6 of \cite{For}. 
\begin{lem}\label{lem9}For
any positive integers $n,k$, $n\geq 2k$, we have $ G_{n-2k,k}=2^{-2k}G_{n}.$\end{lem} 


The following Lemma \ref{lem15} and Lemma \ref{lem16} give the expressions of $G_{n_1,n_2}$ for the cases $n_1$ even and $n_1$ odd separately, where one can express the generalized partition
functions $G_{n_1,n_2}$ in terms of Pfaffians via the method of integration over alternate variables (see \S 6.3.2 in \cite{For}).

\begin{lem}\label{lem15}For
the case $n_1$ even, we have \begin{align*}&{G_{n_1,n_2}}=(n_1!n_2!)c_{n_1+2n_2}[\zeta^{n_1/2}]\Pf[\beta_{j,k}+\zeta\alpha_{j,k}]_{j,k=1,\cdots,n_1+2n_2},
\end{align*}where $[\zeta^{j}]f $ denotes the coefficient of $\zeta^{j} $ in the power series expansion of $f$ and\begin{align*}&\alpha_{j,k}:=\int_{\mathbb{R}^2}\varphi_{k-1}(x)\varphi_{j-1}(y)\sgn(x-y)dxdy,
\\&\beta_{j,k}:=\int_{\mathbb{R}}(\varphi_{k-1}'(x)\varphi_{j-1}(x)-\varphi_{k-1}(x)\varphi_{j-1}'(x))dx.
\end{align*}\end{lem}
\begin{proof}According to \eqref{47}, as in the proof of Proposition 6.3.4 in \cite{For}, applying the method of integration over alternate variables to integrate over $x_1,x_3,\cdots,x_{n_1-1},$ and expanding the resulting determinant to integrate over all the rest variables gives\begin{align*}&{G_{n_1,n_2}}=\frac{(n_1!)c_{n_1+2n_2}}{(n_1/2)!}\sum_{P\in S_{n_1+2n_2}}\varepsilon(P)\prod_{l=1}^{n_1/2}a_{P(2l-1),P(2l)}\prod_{l=n_1/2+1}^{n_1/2+n_2}b_{P(2l-1),P(2l)}
\end{align*}where\begin{align*}&a_{j,k}:=\int_{\mathbb{R}}dx\varphi_{k-1}(x)\int_{-\infty}^xdy\varphi_{j-1}(y),
\\&b_{j,k}:=\int_{\mathbb{R}}\varphi_{k-1}'(x)\varphi_{j-1}(x)dx.
\end{align*}Making the restriction $P(2l)>P(2l-1) ,$ 
 we further have\begin{align*}&{G_{n_1,n_2}}=\frac{(n_1!)c_{n_1+2n_2}}{(n_1/2)!}\sum_{P(2l)>P(2l-1)}\varepsilon(P)\prod_{l=1}^{n_1/2}
\alpha_{P(2l-1),P(2l)}\prod_{l=n_1/2+1}^{n_1/2+n_2}\beta_{P(2l-1),P(2l)}.
\end{align*}Then the result is a consequence
of the definition of a Pfaffian.\end{proof}\begin{lem}\label{lem16}For
the case $n_1$ odd, let $n=n_1+2n_2$, then we have \begin{align*}&\frac{G_{n_1,n_2}}{(n_1!n_2!)c_{n}}=[\zeta^{(n_1-1)/2}]
\Pf\left[\begin{array}{ll}
                                             [\beta_{j,k}+\zeta\alpha_{j,k}]_{j,k=1,\cdots,n}&[\nu_j]_{j=1,\cdots,n} \\
                                             -[\nu_k]_{k=1,\cdots,n}&0
                                           \end{array}
                                     \right],
\end{align*} where $\alpha_{j,k},\ \beta_{j,k} $ are defined in Lemma \ref{lem15} and\begin{align*}&\nu_k:=\int_{\mathbb{R}}\varphi_{k-1}(x)dx.
\end{align*}\end{lem}
\begin{proof}With the same definitions of  $a_{j,k}$ and $b_{j,k}$ as in the proof of Lemma \ref{lem15}, we apply the method of integration over alternate variables again to integrate over $x_1,x_3,\cdots,x_{n_1}$ first, then we expand the resulting determinant and integrate over all the rest variables to get \begin{align*}{G_{n_1,n_2}}=&\frac{(n_1!)c_{n_1+2n_2}}{((n_1-1)/2)!}\sum_{P\in S_{n_1+2n_2}}\varepsilon(P)\nu_{P(n_1)}\times\\&\prod_{l=1}^{(n_1-1)/2}a_{P(2l-1),P(2l)}\prod_{l=1}^{n_2}b_{P(n_1+2l-1),P(n_1+2l)}\\
=&\frac{(n_1!)c_{n_1+2n_2}}{((n_1-1)/2)!}\sum_{P(2l)>P(2l-1);\ l=1,\cdots,(n_1-1)/2+n_2}\varepsilon(P)\nu_{P(n_1+2n_2)}\times\\&\prod_{l=1}^{(n_1-1)/2}\alpha_{P(2l-1),P(2l)}
\prod_{l=(n_1-1)/2}^{(n_1-1)/2+n_2}\beta_{P(2l-1),P(2l)}.
\end{align*}Here, we changed the order $ P(n_1),P(n_1+1),\cdots,P(n_1+2n_2)\to P(n_1+2n_2),P(n_1),\\ \cdots,P(n_1+2n_2-1),$ and made the restriction $P(2l)>P(2l-1) .$ Now we write $\nu_{P(n_1+2n_2)}=\nu_{P(n)}:=\nu_{P(n),n+1}=-\nu_{n+1,P(n)} $ in the above expression, then the result is again a consequence
of the definition of a Pfaffian.\end{proof}Now we need several properties of $\alpha_{j,k},\ \beta_{j,k} $ and $\nu_k$. By \eqref{45}  and \eqref{41}, we first have\begin{align}\label{48}&\beta_{k,k+1}=-\beta_{k+1,k}=\sqrt{2k},\ \beta_{j,k}=0\ \text{for} \ |j-k|\neq 1.
\end{align}We also have the following   \begin{lem}\label{lem17}Let $\alpha_{j,k},\ \beta_{j,k} $ be defined in Lemma \ref{lem15}, $\nu_{k} $ be defined in Lemma \ref{lem16}, and let's define $ \alpha_{0,k}=\alpha_{j,0}=\nu_0=0.$ Then we have

(a) for positive integers $j,k,$ we have \begin{align*}&\sqrt{j-1}\alpha_{j-1,k}-\sqrt{j}\alpha_{j+1,k}=2\sqrt{2}\delta_{jk},\ \sqrt{j-1}\nu_{j-1}-\sqrt{j}\nu_{j+1}=0.
\end{align*}

(b) $\nu_j=0$ for $j$ even;  $\nu_j>0$ for $j$ odd.

(c) $\alpha_{j,k}=\alpha_{1,k}\nu_j/\nu_1$ for $0<j\leq k.$

 (d) If $k$ is odd, then $\alpha_{j,k}=0$ for $0<j\leq k;$ if $k$ is even ($k>0$), then $\alpha_{1,k}>0.$ 
 
  (e) If $n$ is even, $n>0$, $j,k\in\mathbb{Z}\cap[1,n]$, then $$\sum_{l=1}^n\beta_{j,l}\alpha_{l,k}=-4\delta_{jk}.$$\end{lem} \begin{proof}Let's define a skew symmetric inner product $ \langle\cdot|\cdot\rangle_1$ by\begin{align*}&\langle f|g\rangle_1:=\int_{\mathbb{R}^2}g(x)f(y)\sgn(x-y)dxdy,
\end{align*}then we have $$\alpha_{j,k}=\langle \varphi_{j-1}|\varphi_{k-1}\rangle_1 =-\alpha_{k,j}.$$  
Thanks to \eqref{45} and $ \lim\limits_{x\to\pm\infty}\varphi_{j}(x)=0$,  we have\begin{align*}\langle \varphi_{j}'|\varphi_{k}\rangle_1&=\int_{\mathbb{R}}dx\varphi_{k}(x)\left(\int_{-\infty}^x\varphi_{j}'(y)dy-
\int_x^{+\infty}\varphi_{j}'(y)dy\right)\\&=\int_{\mathbb{R}}dx\varphi_{k}(x)(2\varphi_{j}(x))=2\delta_{jk}.
\end{align*}Hence, by \eqref{41}, we will have\begin{align*}2\sqrt{2}\delta_{jk}&=\langle \sqrt{2}\varphi_{j}'|\varphi_{k}\rangle_1=\sqrt{j}\langle \varphi_{j-1}|\varphi_{k}\rangle_1-\sqrt{j+1}\langle \varphi_{j+1}|\varphi_{k}\rangle_1\\&=\sqrt{j}\alpha_{j,k+1}-\sqrt{j+1}\alpha_{j+2,k+1},
\end{align*}and thus we conclude the first identity of (a).

Similarly, we have \begin{align*}0&=\int_{\mathbb{R}}\sqrt{2}\varphi_{j}'(x)dx=
\int_{\mathbb{R}}(\sqrt{j}\varphi_{j-1}(x)-\sqrt{j+1}\varphi_{j+1}(x))dx=\sqrt{j}\nu_{j}-\sqrt{j+1}\nu_{j+2},
\end{align*} which implies the second identity of (a).

If $j$ is even,  we have\begin{align*}&\nu_{j}=\nu_0\prod_{l=0}^{(j-2)/2}\frac{\sqrt{2l}}{\sqrt{2l+1}}=0.
\end{align*}By \eqref{vard}, we have $\varphi_{0}(x)>0 $, and thus $$\nu_1=\int_{\mathbb{R}}\varphi_{0}(x)dx>0, $$ therefore, for $j$ odd, we have \begin{align*}&\nu_{j}=\nu_1\prod_{l=1}^{(j-1)/2}\frac{\sqrt{2l-1}}{\sqrt{2l}}>0.
\end{align*}This shows that (b) is true.

By (a) where $\sqrt{j-1}\alpha_{j-1,k}-\sqrt{j}\alpha_{j+1,k}=0 $ for $0<j<k,$ we will have \begin{align*}&\alpha_{j,k}=\alpha_{0,k}\prod_{l=0}^{(j-2)/2}\frac{\sqrt{2l}}{\sqrt{2l+1}}=0=\alpha_{1,k}\nu_j/\nu_1\ \text{for} \ j\ \text{even},\ 0<j\leq k,\\&\alpha_{j,k}=\alpha_{1,k}\prod_{l=1}^{(j-1)/2}\frac{\sqrt{2l-1}}{\sqrt{2l}}=\alpha_{1,k}\nu_j/\nu_1\ \text{for} \ j\ \text{odd},\ 0<j\leq k,
\end{align*}and thus (c) is true.

Since $\alpha_{j,k}=-\alpha_{k,j} $, we have $\alpha_{k,k}=0. $ If $k$ is odd,  then $0=\alpha_{k,k}=\alpha_{1,k}\nu_k/\nu_1$ and $\nu_1>0\,, \nu_k>0,$ then we must have $\alpha_{1,k}=0 $ and $\alpha_{j,k}=\alpha_{1,k}\nu_j/\nu_1=0 $ for $0<j\leq k$. If $k$ is even ($k>0$), then $k\pm1$ are odd and thus $\alpha_{k,k+1} =0=-\alpha_{k+1,k},\ \nu_{k-1}>0.$ By (a),  we have\begin{align*}&\sqrt{k-1}\alpha_{k-1,k}=\sqrt{k-1}\alpha_{k-1,k}-\sqrt{k}\alpha_{k+1,k}=2\sqrt{2}\delta_{kk}=2\sqrt{2}
\end{align*}and $ \alpha_{1,k}\nu_{k-1}/\nu_1=\alpha_{k-1,k}>0$. Thus we must have $\alpha_{1,k}>0$, which completes (d).

Now we assume that $n$ is even, $n>0$, $j,k\in\mathbb{Z}\cap[1,n]$, then $n+1>k$ and $n+1$ is odd.  By (d), we have $\alpha_{n+1,k}=-\alpha_{k,n+1}=0.$ 
Thus by \eqref{48} and (a), we have \begin{align*}\sum_{l=1}^n\beta_{j,l}\alpha_{l,k}&=\sum_{l=1}^{n+1}\beta_{j,l}\alpha_{l,k}
=-\sqrt{2(j-1)}\alpha_{j-1,k}+\sqrt{2j}\alpha_{j+1,k}
\\&=(-\sqrt{2})\cdot2\sqrt{2}\delta_{jk}=-4\delta_{jk},
\end{align*}which is (e). \end{proof}For the evaluation of Pfaffians, we need the following abstract result.  \begin{lem}\label{lem18}Let $\alpha_{j,k},\ \beta_{j,k} $ be defined for positive integers $j,k$ such that $\alpha_{k,j}=-\alpha_{j,k},\ \beta_{k,j}=-\beta_{j,k}$ and $\beta_{j,k}=0
$ for $|j-k|\neq 1$. Let \begin{align}\label{50}&A_n=[\alpha_{j,k}]_{j,k=1,\cdots,n},\ B_n=[\beta_{j,k}]_{j,k=1,\cdots,n},\ B_n'=\text{diag}(B_{n-1},0)
\end{align}be $n\times n$ antisymmetric matrices. Let's denote $$D_n(\lambda):=\det(B_n+2\lambda I_n),\ D_0(\lambda):=1,$$where $I_n$ is the identity matrix, then we have \begin{align}\label{51}&D_{n+1}(\lambda)=2\lambda D_{n}(\lambda)+\beta_{n,n+1}^2 D_{n-1}(\lambda)\ \text{for}\ n\in\mathbb{Z},\ n>0.
\end{align} If $n>0$ is even, then we have (let's define $\Pf(B_{0}+\lambda A_{0}):=1$)\begin{align}\label{52}&\Pf(B_n+\lambda A_n)=\Pf(B_n'+\lambda A_n)+\beta_{n-1,n}\Pf(B_{n-2}+\lambda A_{n-2}).
\end{align} Moreover, if $n>0$ is even and $B_nA_n=-4I_n$,  then we have \begin{align}\label{53}&\Pf(B_n+\lambda^2 A_n)=D_n(\lambda)/(\Pf B_n)\end{align}and \begin{align}\label{909} \Pf(B_n'+\lambda^2 A_n)=2\lambda D_{n-1}(\lambda)/(\Pf B_n).
\end{align}\end{lem}\begin{proof}The formula \eqref{51} follows from the Laplace expansion of the determinant in the $(n+1)$-th row of $ B_{n+1}+2\lambda I_{n+1} .$ The formula \eqref{52} follows from the Laplace expansion of the Pfaffian (see (6.36) in \cite{For}). Now we assume that $n>0$ is even and $B_nA_n=-4I_n$, then $B_n$ is invertible, $A_n=-4B_n^{-1}$ and\begin{align*}B_n+\lambda^2 A_n&=B_n-4\lambda^2B_n^{-1}=(B_n-2\lambda I_n)B_n^{-1}(B_n+2\lambda I_n)\\&=-(B_n+2\lambda I_n)^TB_n^{-1}(B_n+2\lambda I_n),
\end{align*}here we used the fact that $B_n$ is antisymmetric. By \eqref{49} we have\begin{align*}\Pf(B_n+\lambda^2 A_n)=(-1)^{n/2}\det(B_n+2\lambda I_n)\Pf(B_n^{-1}).
\end{align*}Taking $\lambda=0,$ we have $\Pf(B_n)=(-1)^{n/2}\det(B_n)\Pf(B_n^{-1}). $ Since $B_n$ is invertible, by \eqref{49} again,  we have $\det(B_n)=(\Pf B_n)^2\neq 0$, and thus $(-1)^{n/2}\Pf(B_n^{-1})=(\Pf B_n)^{-1}.$ Therefore, we have \begin{align*}\Pf(B_n+\lambda^2 A_n)=\det(B_n+2\lambda I_n)(\Pf B_n)^{-1}=D_n(\lambda)/(\Pf B_n),
\end{align*}which is \eqref{53}. 
By definition, the above result is also true for $n=0.$ By  definition  of a Pfaffian and the fact that $\beta_{j,k}=0
$ for $|j-k|\neq 1,$ we have \begin{align*}\Pf B_n=\prod_{j=1}^{n/2}\beta_{2j-1,2j},\ \Pf B_n=\beta_{n-1,n}\Pf B_{n-2}.
\end{align*}Combining this with  \eqref{51}, \eqref{52} and \eqref{53}, we have\begin{align*}&\Pf(B_n'+\lambda^2 A_n)=\Pf(B_n+\lambda^2 A_n)-\beta_{n-1,n}\Pf(B_{n-2}+\lambda^2 A_{n-2})\\&=D_n(\lambda)/(\Pf B_n)-\beta_{n-1,n}D_{n-2}(\lambda)/(\Pf B_{n-2})\\&=D_n(\lambda)/(\Pf B_n)-\beta_{n-1,n}^2D_{n-2}(\lambda)/(\Pf B_{n})=2\lambda D_{n-1}(\lambda)/(\Pf B_n),
\end{align*}which is \eqref{909}. This completes the proof.\end{proof}We also need to evaluate the determinant $D_n(\lambda). $\begin{lem}\label{lem19}Let $\beta_{j,k} $ be defined in Lemma \ref{lem15}, i.e. $\beta_{j,k} $ satisfies \eqref{48}. Let's denote $B_n=[\beta_{j,k}]_{j,k=1,\cdots,n}$ and $D_n(\lambda)=\det(B_n+2\lambda I_n)$ with $D_0(\lambda)=1,$ then we have\begin{align*}&D_{n}(\lambda)=\sum_{m=0}^{[n/2]}2^{n-m}{n\choose 2m}\frac{(2m)!}{2^mm!}\lambda^{n-2m}.
\end{align*}
\end{lem}\begin{proof}By \eqref{48} and \eqref{51},  we have\begin{align*}D_{n+1}(\lambda)=2\lambda D_{n}(\lambda)+2n D_{n-1}(\lambda)\ \text{for}\ n\in\mathbb{Z},\ n>0.
\end{align*}Let $\widetilde{H}_n(x)=i^{-n}D_n(ix)$, then we have\begin{align*}\widetilde{H}_{n+1}(x)=2x \widetilde{H}_{n}(x)-2n \widetilde{H}_{n-1}(x)\ \text{for}\ n\in\mathbb{Z},\ n>0.
\end{align*}Moreover, we have $D_0(\lambda)=1,\ B_1=(0),\ D_1(\lambda)=2\lambda;\ \widetilde{H}_0(x)=1,\ \widetilde{H}_1(x)=2x. $ Thus $ \widetilde{H}_n$ satisfy the same iteration formula and initial condition as the Hermite polynomials $ H_n$ (recall \eqref{40}), which implies that $\widetilde{H}_n(x)={H}_n(x). $ By \eqref{54} we have\begin{align*}&D_{n}(\lambda)=i^n\widetilde{H}_n(-i\lambda)=i^nH_n(-i\lambda)\\=&\sum_{m=0}^{[n/2]}i^n(-1)^m2^{n-m}{n\choose 2m}\frac{(2m)!}{2^mm!}(-i\lambda)^{n-2m}\\=&\sum_{m=0}^{[n/2]}2^{n-m}{n\choose 2m}\frac{(2m)!}{2^mm!}\lambda^{n-2m},
\end{align*}which completes the proof.\end{proof}Now we give the proof of Lemma \ref{lem9}.\begin{proof}Let $\alpha_{j,k},\ \beta_{j,k} $ be defined in Lemma \ref{lem15}, $\nu_{k} $ be defined in Lemma \ref{lem16}, and $A_n,\ B_n,\ B_n'$ be defined in \eqref{50}. If $n$ is even, then by (e) of Lemma \ref{lem17},  we have $B_nA_n=-4I_n$. By Lemma \ref{lem15}, Lemma \ref{lem18} and Lemma \ref{lem19}, we have\begin{align*}{G_{n-2k,k}}&=(n-2k)!k!c_{n}[\zeta^{n/2-k}]\Pf[\beta_{j,l}+\zeta\alpha_{j,l}]_{j,l=1,\cdots,n}\\&
=(n-2k)!k!c_{n}[\zeta^{n/2-k}]\Pf(B_n+\zeta A_n)\\&
=(n-2k)!k!c_{n}[\zeta^{n-2k}]\Pf(B_n+\zeta^2 A_n)\\&
=(n-2k)!k!c_{n}[\zeta^{n-2k}]D_n(\zeta)/(\Pf B_n)\\&
=(n-2k)!k!c_{n}2^{n-k}{n\choose 2k}\frac{(2k)!}{2^kk!}(\Pf B_n)^{-1}\\&
=c_{n}2^{n-2k}n!(\Pf B_n)^{-1},
\end{align*}and thus$${G_{n-2k,k}}=2^{-2k}G_{n,0}=2^{-2k}G_{n}. $$

If $n$ is odd, by Lemma \ref{lem16}, we first have\begin{align*}&\frac{G_{n-2k,k}}{(n-2k)!k!c_{n}}=[\zeta^{(n-2k-1)/2}]
\Pf\left[\begin{array}{ll}
                                             [\beta_{j,l}+\zeta\alpha_{j,l}]_{j,l=1,\cdots,n}&[\nu_j]_{j=1,\cdots,n} \\
                                             -[\nu_l]_{l=1,\cdots,n}&0
                                           \end{array}
                                     \right].
\end{align*}By Lemma \ref{lem17}, we also have $B_{n+1}A_{n+1}=-4I_{n+1}$, $\alpha_{j,n+1}=\alpha_{1,n+1}\nu_j/\nu_1=-\alpha_{n+1,j}$ for $0<j\leq n,$ and $\alpha_{1,n+1}>0,\ \nu_1>0. $ By   definition, $\Pf X$ is linear with respect to the last row of $X,$ thus for $ \lambda:=\zeta\alpha_{1,n+1}/\nu_1$,  we  have\begin{align*}&\lambda
\Pf\left[\begin{array}{ll}
                                             [\beta_{j,l}+\zeta\alpha_{j,l}]_{j,l=1,\cdots,n}&[\nu_j]_{j=1,\cdots,n} \\
                                             -[\nu_l]_{l=1,\cdots,n}&0
                                           \end{array}
                                     \right]\\=&
\Pf\left[\begin{array}{ll}
                                             [\beta_{j,l}+\zeta\alpha_{j,l}]_{j,l=1,\cdots,n}&\lambda[\nu_j]_{j=1,\cdots,n} \\
                                             -\lambda[\nu_l]_{l=1,\cdots,n}&0
                                           \end{array}
                                     \right]\\=&
\Pf\left[\begin{array}{ll}
                                             [\beta_{j,l}+\zeta\alpha_{j,l}]_{j,l=1,\cdots,n}&[\zeta\alpha_{j,n+1}]_{j=1,\cdots,n} \\
                                             \ [\zeta\alpha_{n+1,l}]_{l=1,\cdots,n}&0
                                           \end{array}
                                     \right]\\=&\Pf(B_{n+1}'+\zeta A_{n+1}),
\end{align*}where $B_{n+1}'=\text{diag}(B_{n},0)$. Hence,  by  Lemma \ref{lem18} and Lemma \ref{lem19}, we have \begin{align*}&\frac{\alpha_{1,n+1} G_{n-2k,k}}{\nu_1(n-2k)!k!c_{n}}\\=&\frac{\alpha_{1,n+1}}{\nu_1}[\zeta^{(n-2k-1)/2+1}]
\zeta\Pf\left[\begin{array}{ll}
                                             [\beta_{j,l}+\zeta\alpha_{j,l}]_{j,l=1,\cdots,n}&[\nu_j]_{j=1,\cdots,n} \\
                                             -[\nu_l]_{l=1,\cdots,n}&0
                                           \end{array}
                                     \right]\\=&[\zeta^{(n-2k-1)/2+1}]\lambda
\Pf\left[\begin{array}{ll}
                                             [\beta_{j,l}+\zeta\alpha_{j,l}]_{j,l=1,\cdots,n}&[\nu_j]_{j=1,\cdots,n} \\
                                             -[\nu_l]_{l=1,\cdots,n}&0
                                           \end{array}
                                     \right]\\=&
                                     [\zeta^{(n-2k-1)/2+1}]\Pf(B_{n+1}'+\zeta A_{n+1})\\=&[\zeta^{n-2k+1}]\Pf(B_{n+1}'+\zeta^2 A_{n+1})\\=&[\zeta^{n-2k+1}](2\zeta D_n(\zeta)/(\Pf B_{n+1}))\\=&2[\zeta^{n-2k}] D_n(\zeta)/(\Pf B_{n+1})\\=&\frac{2^{n-k+1}}{\Pf B_{n+1}}{n\choose 2k}\frac{(2k)!}{2^kk!}.
\end{align*}Therefore, we have$$ G_{n-2k,k}=\frac{2^{n-2k+1}n!\nu_1c_{n}}{\alpha_{1,n+1}\Pf B_{n+1}},$$
which implies $${G_{n-2k,k}}=2^{-2k}G_{n,0}=2^{-2k}G_{n}.$$
 This completes the whole proof of Lemma \ref{lem9}.\end{proof}


\section{Auxiliary point processes}\label{aux}
We need to introduce two more auxiliary point processes to derive the main result.  First, instead of $\chi^{(n)}$ (recall \eqref{chiii}),
it is more convenient to consider the point process  defined as \begin{align*}&\widetilde{\chi}^{(n)}=\sum_{i< j}\delta_{n|\lambda_i-\lambda_j|}=\sum_{\lambda_i>\lambda_j}\delta_{n(\lambda_i-\lambda_j)}.
\end{align*}Then we have $$\chi^{(n)}\leq\widetilde{\chi}^{(n)}, $$ in fact,  we can write $$\widetilde{\chi}^{(n)}=\sum\limits_{j=1}^{n-1}\widetilde{\chi}^{(n,j)} $$such that\begin{align*}&\widetilde{\chi}^{(n,j)}=\sum_{i=1}^{n-j}\delta_{n(\lambda_{(i+j)}-\lambda_{(i)})}.
\end{align*}For any Borel set $ B\subset\mathbb{R}$, we have $$\widetilde{\chi}^{(n,1)}= {\chi}^{(n)} \,\,\mbox{and}\,\, 0\leq\widetilde{\chi}^{(n,j)}(B)\leq n .$$ For the auxiliary point process $\widetilde\chi^{(n)}\geq \chi^{(n)}$,  we will prove that $\widetilde\chi^{(n)}- \chi^{(n)}\to 0$ as $n\to \infty$ almost surely (see Lemma \ref{lem4}), which indicates that there is no successive small gaps.

We now introduce another auxiliary point process as \begin{align*}&\rho^{(k,n)}=\sum_{i_1,\cdots,i_{2k}\ \text{all distinct},\ i_{2j-1}<i_{2j} }\delta_{(n|\lambda_{i_1}-\lambda_{i_2}|,\cdots,n|\lambda_{i_{2k-1}}-\lambda_{i_{2k}}|)}.
\end{align*}  
The following lemma gives the estimates of $\rho^{(k,n)}$ in terms of $\widetilde{\chi}^{(n)},$ and we will see that 
$\rho^{(k,n)} $ is basically equivalent to the factorial moment of $\widetilde{\chi}^{(n)} $ (see \eqref{25}).
\begin{lem}\label{lem14} For any bounded interval  $ A\subset\mathbb{R}_{+}$, we have \begin{equation}\label{leaa}\rho^{(k,n)}(A^k)\leq\frac{(\widetilde{\chi}^{(n)}(A))!}{(\widetilde{\chi}^{(n)}(A)-k)!}.
\end{equation}Given $c_1$ such that $A\subset(0,c_1)$, let's denote $c_n=c_1n^{-1}$ and \begin{equation}a=\max\{i-j:i,j\in\mathbb{Z}\cap[1,n],\ \lambda_{(i)}-\lambda_{(j)}< 2c_n \}.
\end{equation}If $c_n\in (0,1)$, then we have\begin{equation}\label{aaad}0\leq\frac{(\widetilde{\chi}^{(n)}(A))!}{(\widetilde{\chi}^{(n)}(A)-k)!}-\rho^{(k,n)}(A^k)
\leq k(k-1)(a-1)(\widetilde{\chi}^{(n)}(A))^{k-1}
\end{equation}and\begin{equation}\label{add}\rho^{(k,n)}(A^k)\geq(\widetilde{\chi}^{(n)}(A))^k-k(k-1)a(\widetilde{\chi}^{(n)}(A))^{k-1}.
\end{equation}Moreover, let $A_1=(0,2c_1)$, then we have \begin{equation}\label{last}\rho^{(k,n)}(A_1^k)\geq\frac{(a+1)!}{(a+1-2k)!2^k}.
\end{equation}\end{lem}\begin{proof}Let's denote \begin{align*}&X_1=\{(i_1,\cdots,i_{2k}):i_j\in\mathbb{Z}, 1\leq i_j\leq n,\ \forall\ 1\leq j\leq 2k,\\ &i_{2j-1}< i_{2j},\ \forall\ 1\leq j\leq k,\ \ \{i_{2j-1}, i_{2j}\}\neq \{i_{2l-1}, i_{2l}\}, \ \forall\ 1\leq j<l\leq k\},\\&X_2=\{(i_1,\cdots,i_{2k}):i_j\in\mathbb{Z}, 1\leq i_j\leq n,\ \forall\ 1\leq j\leq 2k,\\ &i_{2j-1}< i_{2j},\ \forall\ 1\leq j\leq k,\ \ i_{j}\neq i_{l},\  \forall\ 1\leq j<l\leq 2k\},\\&Y_{j,l}=\{(i_1,\cdots,i_{2k})\in X_1:\{i_{2j-1}, i_{2j}\}\cap \{i_{2l-1}, i_{2l}\}\neq \emptyset\},
\end{align*}then we have $$X_2\subseteq X_1\,\,\,\mbox{and} \,\,\,X_1\setminus X_2=\cup_{1\leq j<l\leq k}Y_{j,l}.$$ Let\begin{align*}&X_{m,A}=\{(i_1,\cdots,i_{2k})\in X_m:n|\lambda_{i_{2j-1}}-\lambda_{i_{2j}}|\in A,\ \forall\ 1\leq j\leq k\},\ m=1,2,\\&Y_{m,l,A}=\{(i_1,\cdots,i_{2k})\in Y_{m,l}:n|\lambda_{i_{2j-1}}-\lambda_{i_{2j}}|\in A,\ \forall\ 1\leq j\leq k\},
\end{align*}then we have $$\rho^{(k,n)}(A^k)=|X_{2,A}|,\  X_{2,A}\subseteq X_{1,A}\,\,\,\mbox{and}\,\,\,|X_{1,A}|= \dfrac{(\widetilde{\chi}^{(n)}(A))!}{(\widetilde{\chi}^{(n)}(A)-k)!},$$ which implies \eqref{leaa}, here $|X|$ is  cardinality  of the set $X.$

 We also have $X_{1,A}\setminus X_{2,A}=\cup_{1\leq j<l\leq k}Y_{j,l,A}$ and $|Y_{j,l,A}|=|Y_{1,2,A}| $ for $ 1\leq j<l\leq k$ by symmetry.  Therefore, we have \begin{align}\label{38}&|X_{1,A}|-| X_{2,A}|\leq\sum_{1\leq j<l\leq k}|Y_{j,l,A}|=k(k-1)|Y_{1,2,A}|/2.
\end{align}

If $a=0,$ then we have $n|\lambda_{j}-\lambda_{l}|\geq n^{}(2c_n)=2c_1 $ for every $1\leq j<l\leq n,$ i.e., $n|\lambda_{j}-\lambda_{l}|\not\in A$, and thus $\widetilde{\chi}^{(n)}(A)=\rho^{(k,n)}(A^k)=0; $ if $k=1$ then $\widetilde{\chi}^{(n)}(A)=\rho^{(1,n)}(A) $ by definitions. In both cases,  the inequalities \eqref{aaad} and \eqref{add} are clearly true, thus we only need to consider the case $a>0,k>1.$

Let $ \lambda_{i,j}:=n|\lambda_{i}-\lambda_{j}|.$ For fixed $\lambda_{i_{1},i_{2}}\in A, $ let \begin{align*}&T_j=\{l:l\neq i_j,n|\lambda_{i_{l}}-\lambda_{i_{j}}|\in A\},\\ &T_j'=\{l:l\neq i_j,|\lambda_{i_{l}}-\lambda_{i_{j}}|\in (0,c_n)\},\ j=1,2.
\end{align*}Then we have $T_j\subseteq T_j'$ because $n|\lambda_{i_{l}}-\lambda_{i_{j}}|\in A $ implies $|\lambda_{i_{l}}-\lambda_{i_{j}}|\in n^{-1}A\subset n^{-1}(0,c_1) =(0,c_n) .$ Let's assume $ \lambda_{i_1}=\lambda_{(p)}$, then we have\begin{align*}\{\lambda_l:l\in T_1'\cup \{i_1\}\}&=\{\lambda_{(q)}:|\lambda_{(q)}-\lambda_{(p)}|<c_n\}\\&=\{\lambda_{(q)}:r\leq q\leq s\},
\end{align*}for some $r,s\in \mathbb{Z}\cap[1,n]$ such that $|\lambda_{(r)}-\lambda_{(p)}|<c_n,\ |\lambda_{(s)}-\lambda_{(p)}|<c_n.$ Therefore, we have $|\lambda_{(r)}-\lambda_{(s)}|<2c_n$ and $s-r\leq a$  by the definition of $a$.  Since $i_1\not\in T_1'$, we have\begin{align*}|T_1'|+1&=|\{\lambda_l:l\in T_1'\cup \{i_1\}\}|=|\{\lambda_{(q)}:r\leq q\leq s\}|\\&\leq s-r+1\leq a+1,
\end{align*}and thus $|T_1|\leq |T_1'|\leq a$. Similarly  we have $|T_2|\leq |T_2'|\leq a.$

Now for $\lambda_{i_{1}, i_{2}}\in A, $ by definition we have $ i_2\in T_1$ and $i_1\in T_2.$ If $\lambda_{i_{3}, i_{4}}\in A,\ i_{3}<i_{4},\ \{i_{1}, i_{2}\}\cap \{i_{3}, i_{4}\}\neq \emptyset,\ \{i_{1}, i_{2}\}\neq \{i_{3}, i_{4}\}$, then we must have $\{i_{3}, i_{4}\}=\{i_{1}, l\},\ l\in T_2\setminus\{i_1\} $ or $\{i_{3}, i_{4}\}=\{i_{2}, l\},\ l\in T_1\setminus\{i_2\}$.  Thus for $\lambda_{i_{1}, i_{2}}\in A, $ the number of $(i_3,i_4)$ satisfying $\lambda_{i_{3}, i_{4}}\in A,\ i_{3}<i_{4}, \ \{i_{1}, i_{2}\}\cap \{i_{3}, i_{4}\}\neq \emptyset,\ \{i_{1}, i_{2}\}\neq \{i_{3}, i_{4}\}$ is at most $|T_2\setminus\{i_1\}|+|T_1\setminus\{i_2\}|=|T_2|-1+|T_1|-1\leq 2(a-1). $ Now there are $\widetilde{\chi}^{(n)}(A) $ choices of $(i_{1},i_{2});$ for fixed $(i_{1},i_{2}) $, there are at most $2(a-1) $ choices of $(i_3,i_4) $ and $\widetilde{\chi}^{(n)}(A) $ choices of $(i_{2l-1},i_{2l})$ with $3\leq l\leq k$ to satisfy $ (i_1,\cdots,i_{2k})\in Y_{1,2,A} ,$ thus we have \begin{align*}|Y_{1,2,A}|\leq \widetilde{\chi}^{(n)}(A)\times2(a-1)\times\widetilde{\chi}^{(n)}(A)^{k-2}=2(a-1)\widetilde{\chi}^{(n)}(A)^{k-1}.
\end{align*}Therefore, by \eqref{38} we have \begin{align*}0&\leq\frac{(\widetilde{\chi}^{(n)}(A))!}{(\widetilde{\chi}^{(n)}(A)-k)!}-\rho^{(k,n)}(A^k)\\&=|X_{1,A}|-| X_{2,A}|\\&\leq k(k-1)|Y_{1,2,A}|/2
\\&\leq k(k-1)(a-1)(\widetilde{\chi}^{(n)}(A))^{k-1},
\end{align*}which is \eqref{aaad}. The inequality \eqref{add} follows from \eqref{aaad} and the fact that\begin{align*}\frac{(\widetilde{\chi}^{(n)}(A))!}{(\widetilde{\chi}^{(n)}(A)-k)!}
=&\prod_{j=0}^{k-1}(\widetilde{\chi}^{(n)}(A)-j)=(\widetilde{\chi}^{(n)}(A))^k\prod_{j=0}^{k-1}(1-j/\widetilde{\chi}^{(n)}(A))\\ \geq& (\widetilde{\chi}^{(n)}(A))^k\left(1-\sum_{j=0}^{k-1}j/\widetilde{\chi}^{(n)}(A)\right)\\=&(\widetilde{\chi}^{(n)}(A))^k-k(k-1) (\widetilde{\chi}^{(n)}(A))^{k-1}/2.
\end{align*}
To prove  \eqref{last}, by the definition of $a$, there exists $r,s\in \mathbb{Z}\cap[1,n]$ such that $|\lambda_{(r)}-\lambda_{(s)}|<2c_n ,\ s-r=a.$ Let $$Z=\{j:\lambda_{j}=\lambda_{(q)},\ r\leq q\leq s\}$$ and\begin{align*}&X_3=\{(i_1,\cdots,i_{2k}):i_j\in {Z},\ \forall\ 1\leq j\leq 2k,\\ &i_{2j-1}< i_{2j},\ \forall\ 1\leq j\leq k,\ \ i_{j}\neq i_{l},\  \forall\ 1\leq j<l\leq 2k\}.
\end{align*}Then we have $ |Z|=s-r+1=a+1,\ X_3\subseteq X_2$ and\begin{align*}&|X_3|=\frac{|Z|!}{(|Z|-2k)!2^k}=\frac{(a+1)!}{(a+1-2k)!2^k}.
\end{align*}Moreover, we have $|\lambda_{j}-\lambda_{l}|\leq|\lambda_{(r)}-\lambda_{(s)}|<2c_n$  for $j\,,l\in Z$.  For $(i_1,\cdots,i_{2k})\in X_3$, we have $ 0<n|\lambda_{i_{2j-1}}-\lambda_{i_{2j}}|<2nc_n=2c_1$ for $1\leq j\leq k$, i.e., $n|\lambda_{i_{2j-1}}-\lambda_{i_{2j}}|\in(0,2c_1)=A_1$. Therefore, we have $X_3\subseteq X_{2,A_1}$ and thus \begin{align*}&\rho^{(k,n)}(A_1^k)=|X_{2,A_1}|\geq|X_3|=\frac{(a+1)!}{(a+1-2k)!2^k},
\end{align*}which is \eqref{last}. This completes the whole proof.\end{proof}
\section{No successive small gaps}\label{nssg}

In this section,  we will prove the following lemma which indicates that there is no successive small gaps. 
 \begin{lem}\label{lem4} For any bounded interval  $ A\subset\mathbb{R}_{+}$,  we have $\chi^{(n)}(A)-\widetilde{\chi}^{(n)}(A)\to0$ in probability  as $n\to+\infty$.\end{lem}
To prove Lemma \ref{lem4},  we will need  the upper and lower bounds in the following  integral lemma.  \begin{lem}\label{lem10}Let's assume $\lambda_j\in\mathbb{R}$ (not necessarily distinct) for $1\leq j\leq m$,  $m$ and $n$
are positive integers such that $m< n$,  and $2n c^2\in(0,1)$ with $c>0$.  Let's denote \begin{equation}F(x):=e^{-x^2/2}\prod_{j=1}^m(x-\lambda_j),\end{equation} then we have\begin{equation}\label{fdt} \int_{\mathbb{R}}|F'(x)|^2dx\leq(2n)\int_{\mathbb{R}}|F(x)|^2dx\end{equation}and \begin{equation}\label{slt}\begin{split}(1-n c^2){c^{2}}\int_{\mathbb{R}}d x_1|F(x_1)|^{2}&\leq 
\int_{\mathbb{R}}d x_1\int_{x_1-c}^{x_1+c}d x_2|x_1-x_2||F(x_1)||F(x_2)| \\
& \leq {c^{2}}\int_{\mathbb{R}}d x_1|F(x_1)|^{2}.\end{split}
\end{equation}Moreover, given an interval $A\subset(0,c),$ let's denote $A_1=A\cup(-A)$ and \begin{equation}\varphi(A):=\int_Audu, \end{equation} then we have\begin{equation}\label{sdsss}\begin{split}
(1-n c^2)\cdot2\varphi(A)\int_{\mathbb{R}}d x_1|F(x_1)|^{2} &\leq\int_{\mathbb{R}}d x_1\int_{x_1+A_1}d x_2|x_1-x_2||F(x_1)||F(x_2)|\\
&\leq  2\varphi(A)\int_{\mathbb{R}}d x_1|F(x_1)|^{2}.\end{split}
\end{equation} Given $B=\cup_{i=1}^m(\lambda_i,\lambda_i+c)^2\subset\mathbb{R}^2$, we have\begin{equation}\label{lrime}
\int_{B}|x_1-x_2||F(x_1)||F(x_2)|d x_1d x_2\leq nc^{4}\int_{\mathbb{R}}|F(x)|^2dx.
\end{equation}\end{lem}\begin{proof}Note that $F(x)\in V_m$ (see \eqref{44}), therefore, we can write\begin{align*}&
F(x)=\sum_{j=0}^m a_j\varphi_j(x).
\end{align*}By \eqref{41} we have\begin{align*} F'(x)&=\sum_{j=0}^m\frac{a_j}{\sqrt{2}}(\sqrt{j}\varphi_{j-1}(x)-\sqrt{j+1}\varphi_{j+1}(x))\\
&=\sum_{j=0}^{m+1}\frac{\sqrt{j+1}a_{j+1}-\sqrt{j}a_{j-1}}{\sqrt{2}}\varphi_{j}(x),
\end{align*}where $\varphi_{-1}(x)=0,\ a_{-1}=a_{m+1}=a_{m+2}=0. $ By \eqref{45} we have $$ \int_{\mathbb{R}}|F(x)|^2dx=\sum_{j=0}^m|a_j|^2$$ and  $$\int_{\mathbb{R}}|F'(x)|^2dx=\sum_{j=0}^{m+1}\left|\frac{\sqrt{j+1}a_{j+1}-\sqrt{j}a_{j-1}}{\sqrt{2}}\right|^2.$$Using $|a+b|^2\leq 2(|a|^2+|b|^2)$ and $a_{-1}=a_{m+1}=a_{m+2}=0 $, we have \begin{align*} \int_{\mathbb{R}}|F'(x)|^2dx&\leq\sum_{j=0}^{m+1}(|\sqrt{j+1}a_{j+1}|^2+|\sqrt{j}a_{j-1}|^2)\\
&=\sum_{j=1}^{m+2}j|a_{j}|^2+\sum_{j=-1}^{m}(j+1)|a_{j}|^2=\sum_{j=0}^{m}(2j+1)|a_{j}|^2\\ &\leq \sum_{j=0}^{m}(2m+1)|a_{j}|^2=(2m+1)\int_{\mathbb{R}}|F(x)|^2dx\\&\leq(2n)\int_{\mathbb{R}}|F(x)|^2dx,
\end{align*} which is the first inequality \eqref{fdt}. Here we used the fact that $m<n,\ n\geq m+1.$

To prove \eqref{slt},
 a change of variables $x_2=x_1+t$ yields \begin{align}\label{37}&\int_{\mathbb{R}}d x_1\int_{x_1-c}^{x_1+c}d x_2|x_1-x_2||F(x_1)||F(x_2)| \\ \nonumber
=&\int_{-c}^c|t|dt\int_{\mathbb{R}}|F(x_1)||F(x_1+t)|d x_1.
\end{align}We also have \begin{equation}\label{llldd}\begin{split} & \int_{\mathbb{R}}\left||F(x_1)|-|F(x_1+t)|\right|^2d x_1\\=&\int_{\mathbb{R}}(|F(x_1)|^{2}+|F(x_1+t)|^{2})d x_1-2\int_{\mathbb{R}}|F(x_1)||F(x_1+t)|d x_1\\=&2\int_{\mathbb{R}}|F(x_1)|^{2}d x_1-2\int_{\mathbb{R}}|F(x_1)||F(x_1+t)|d x_1,\end{split}
\end{equation}which implies\begin{align}\label{36}&\int_{\mathbb{R}}|F(x_1)||F(x_1+t)|dx_1 \leq\int_{\mathbb{R}}|F(x_1)|^{2}d x_1.
\end{align}
By \eqref{37} and \eqref{36}, we have\begin{align*}&\int_{\mathbb{R}}d x_1\int_{x_1-c}^{x_1+c}d x_2|x_1-x_2||F(x_1)||F(x_2)|\\
\leq&\int_{-c}^c|t|dt\int_{\mathbb{R}}|F(x_1)|^{2}d x_1=c^2\int_{\mathbb{R}}|F(x_1)|^{2}d x_1,
\end{align*}which is the upper bound in \eqref{slt}.

On the other hand, we have\begin{align*}&\left||F(x_1)|-|F(x_1+t)|\right|^2\leq \left|F(x_1)-F(x_1+t)\right|^2\\=&\left|-t\int_0^1F'(x_1+ts)ds\right|^2\leq |t|^2\int_0^1|F'(x_1+ts)|^2ds,
\end{align*} and thus by \eqref{fdt} we have \begin{align*}&\int_{\mathbb{R}}\left||F(x_1)|-|F(x_1+t)|\right|^2d x_1\\ \leq &|t|^2\int_{\mathbb{R}}\int_0^1|F'(x_1+ts)|^2dsd x_1\\=&|t|^2\int_0^1[\int_{\mathbb{R}}|F'(x_1+ts)|^2d x_1]ds=|t|^2\int_0^1\int_{\mathbb{R}}|F'(x_1)|^2d x_1ds\\ =&|t|^2\int_{\mathbb{R}}|F'(x_1)|^2d x_1\leq 2n|t|^2\int_{\mathbb{R}}|F(x_1)|^2d x_1.\end{align*}  Combining this estimate with identity \eqref{llldd}, we have $$\int_{\mathbb{R}}|F(x_1)||F(x_1+t)|dx_1 \geq(1-n|t|^2)\int_{\mathbb{R}}|F(x_1)|^{2}d x_1,\ \forall\ t\in(-c,c),$$  and thus the uniform lower bound\begin{equation}\label{34}\int_{\mathbb{R}}|F(x_1)||F(x_1+t)|dx_1 \geq(1-nc^2)\int_{\mathbb{R}}|F(x_1)|^{2}d x_1,\ \forall\ t\in(-c,c).
\end{equation} Therefore, combining \eqref{37} and \eqref{34}, we can conclude the lower bound  in \eqref{slt}.

Notice that\begin{align*}&
\int_{\mathbb{R}}d x_1\int_{x_1+A_1}d x_2|x_1-x_2||F(x_1)||F(x_2)|\\
=&\int_{A_1}|t|dt\int_{\mathbb{R}}|F(x_1)||F(x_1+t)|d x_1,
\end{align*}then \eqref{sdsss} follows from \eqref{36}, \eqref{34} and the fact that\begin{align*}&
\int_{A_1}|t|dt=2\int_{A}tdt=2\varphi(A).
\end{align*}Let $B_1=\cup_{i=1}^m(\lambda_i,\lambda_i+c)\subset\mathbb{R}$, then for $(x_1,x_2)\in B=\cup_{i=1}^m(\lambda_i,\lambda_i+c)^2$, we have $x_1,x_2\in B_1,\ (x_2,x_1)\in B,\ |x_1-x_2|\leq c,$ and thus we first have \begin{align*}&
\int_{B}|x_1-x_2||F(x_1)||F(x_2)|d x_1d x_2\leq \frac{1}{2}\int_{B}|x_1-x_2|(|F(x_1)|^2+|F(x_2)|^2)d x_1d x_2\\=&\int_{B}|x_1-x_2||F(x_1)|^2d x_1d x_2\leq\int_{B_1}dx_1\int_{x_1-c}^{x_1+c}dx_2|x_1-x_2||F(x_1)|^2\\=&c^2\int_{B_1}|F(x_1)|^2dx_1.
\end{align*}Without loss of generality we can assume that $ \lambda_1\leq \cdots\leq\lambda_m$ and let's denote $I_j=(\lambda_j,\lambda_j+c)\cap(\lambda_j,\lambda_{j+1}]$ for $1\leq j<m,$ $I_m=(\lambda_m,\lambda_m+c)$, then we have $B_1=\cup_{j=1}^mI_j$ and $ I_j\cap I_k=\emptyset$ for $j\neq k.$ By definition we have $F(\lambda_j)=0$ and\begin{align*}&
|F(z)|^2=\left|\int_{\lambda_j}^zF'(x)dx\right|^2\leq |z-\lambda_j|\int_{\lambda_j}^z|F'(x)|^2dx\leq |z-\lambda_j|\int_{I_j}|F'(x)|^2dx
\end{align*}for $z\in I_j\subseteq (\lambda_j,\lambda_j+c).$ Thus we have \begin{align*}
\int_{I_j}|F(z)|^2dz\leq &\int_{I_j}|z-\lambda_j|dz\int_{I_j}|F'(x)|^2dx\\ \leq& \int_{(\lambda_j,\lambda_j+c)}|z-\lambda_j|dz\int_{I_j}|F'(x)|^2dx\\= &(c^2/2)\int_{I_j}|F'(x)|^2dx.
\end{align*}Combining this with \eqref{fdt}, we further have \begin{align*}&
\int_{B}|x_1-x_2||F(x_1)||F(x_2)|d x_1d x_2\\ \leq & c^2\int_{B_1}|F(x_1)|^2dx_1=c^2\sum_{j=1}^m\int_{I_j}|F(x_1)|^2dx_1\\ \leq & c^2\sum_{j=1}^m(c^2/2)\int_{I_j}|F'(x)|^2dx=(c^4/2)\int_{B_1}|F'(x)|^2dx\\ \leq&(c^4/2)\int_{\mathbb{R}}|F'(x)|^2dx
\leq (c^4/2)(2n)\int_{\mathbb{R}}|F(x)|^2dx=nc^4\int_{\mathbb{R}}|F(x)|^2dx,
\end{align*}
which is \eqref{lrime}. This completes the proof.\end{proof}\subsection{No successive small gaps}
Now we can prove Lemma \ref{lem4}. We first need the following lemma which gives more precise meaning that there is no successive small gaps.\begin{lem}\label{lem11} For $A=(0,c_0)$ and $n>2c_0^2+2$, we have \begin{align*}&\mathbb{P}(\widetilde{\chi}^{(n,2)}(A)>0)\leq c^4_0/(8n).
\end{align*}\end{lem}\begin{proof}If $\widetilde{\chi}^{(n,2)}(A)>0 $, then there exist distinct $i,j,k$ such that $ \lambda_j,\lambda_k\in(\lambda_i,\lambda_{i}+c_0/n).$ Let's denote \begin{align*}\Lambda_{j,k,c}:=\{&(\lambda_1,\cdots, \lambda_n):\exists\ i\in\mathbb{Z}\cap[1,n]\ s.t.\ \lambda_j,\lambda_k\in(\lambda_i,\lambda_{i}+c)\},
\end{align*}then we have\begin{align*}&\mathbb{P}(\widetilde{\chi}^{(n,2)}(A)>0)\\ \leq &\sum_{1\leq j<k\leq n}\mathbb{P}((\lambda_1,\cdots, \lambda_n)\in\Lambda_{j,k,c_0/n})\\=&\mathbb{P}((\lambda_1,\cdots, \lambda_n)\in\Lambda_{n-1,n,c_0/n})n(n-1)/2.
\end{align*}For fixed $(\lambda_1,\cdots, \lambda_{n-2})\in\mathbb{R}^{n-2},\ c>0, $ as in Lemma \ref{lem10}, let's denote\begin{align*}&B(\lambda_1,\cdots, \lambda_{n-2},c):=\cup_{i=1}^{n-2}(\lambda_i,\lambda_i+c)^2\subset\mathbb{R}^2,
\end{align*}then $(\lambda_1,\cdots, \lambda_n)\in\Lambda_{n-1,n,c} $ is equivalent to $(\lambda_{n-1}, \lambda_n)\in B(\lambda_1,\cdots, \lambda_{n-2},c). $

With $c=c_0/n>0,$ we have $2nc^2=2c_0^2/n\in(0,1)$ by assumption,  then by \eqref{lrime}, we have \begin{align*}&
\int_{B(\lambda_1,\cdots, \lambda_{n-2},c_0/n)}|x_1-x_2||F(x_1)||F(x_2)|d x_1d x_2\leq n(c_0/n)^{4}\int_{\mathbb{R}}|F(x)|^2dx,
\end{align*}where $$F(x)=e^{-x^2/2}\prod_{j=1}^{n-2}(x-\lambda_j).$$ Hence,  we have \begin{align*}&\mathbb{P}((\lambda_1,\cdots, \lambda_n)\in\Lambda_{n-1,n,c_0/n})\\=&\frac{1}{G_n}\int_{\Lambda_{n-1,n,c_0/n}}e^{-\sum\limits_{i=1}^n\lambda_i^2/2}\prod_{1\leq i<j\leq n}|\lambda_i-\lambda_j| d\lambda_1\cdots d\lambda_n\\=&\frac{1}{G_n}\int_{\mathbb{R}^{n-2}}d\lambda_1\cdots d\lambda_{n-2}e^{-\sum\limits_{i=1}^{n-2}\lambda_i^2/2}
\prod_{1\leq j<k\leq n-2}|\lambda_j-\lambda_k|\\&\times\int_{B(\lambda_1,\cdots, \lambda_{n-2},c_0/n)}|x_1-x_2|e^{-x_1^2/2-x_2^2/2}\prod_{i=1}^{2}\prod_{j=1}^{n-2}|x_i-\lambda_j|d x_1d x_2\\=&\frac{1}{G_n}\int_{\mathbb{R}^{n-2}}d\lambda_1\cdots d\lambda_{n-2}e^{-\sum\limits_{i=1}^{n-2}\lambda_i^2/2}
\prod_{1\leq j<k\leq n-2}|\lambda_j-\lambda_k|\\&\times \int_{B(\lambda_1,\cdots, \lambda_{n-2},c_0/n)}|x_1-x_2||F(x_1)||F(x_2)|d x_1d x_2 \\ \leq&\frac{n(c_0/n)^{4}}{G_n}\int_{\mathbb{R}^{n-2}}d\lambda_1\cdots d\lambda_{n-2}e^{-\sum\limits_{i=1}^{n-2}\lambda_i^2/2}
\prod_{1\leq j<k\leq n-2}|\lambda_j-\lambda_k|  \\&\times\int_{\mathbb{R}}e^{-x^2}\prod_{j=1}^{n-2}|x-\lambda_j|^2d x\\=&\frac{n(c_0/n)^{4}}{G_n}G_{n-2,1}=\frac{n(c_0/n)^{4}}{4},
\end{align*}where  we used Lemma \ref{lem9} with $k=1$ in the last step. Therefore, we have
\begin{align*}\mathbb{P}(\widetilde{\chi}^{(n,2)}(A)>0)\leq& \mathbb{P}((\lambda_1,\cdots, \lambda_n)\in\Lambda_{n-1,n,c_0/n})n(n-1)/2\\ \leq&\frac{n(c_0/n)^{4}}{4}n^2/2=\frac{c^4_0}{8n}.
\end{align*}This completes the proof. \end{proof}Now we can give the proof of Lemma \ref{lem4} using Lemma \ref{lem11}.\begin{proof}Let $c_0$ be such that $A\subset(0,c_0)$ and $A_1=(0,c_0).$ Then $\chi^{(n)}(A)-\widetilde{\chi}^{(n)}(A)\neq 0$ implies $\widetilde{\chi}^{(n,j)}(A)>0 $ for some $j>1$ and thus we must have  $\widetilde{\chi}^{(n,2)}(A_1)\geq\widetilde{\chi}^{(n,j)}(A_1)\geq \widetilde{\chi}^{(n,j)}(A)>0$. For  $ n>2c_0^2+2,$ by Lemma \ref{lem11} we deduce that\begin{align*}&\mathbb{P}(\chi^{(n)}(A)-\widetilde{\chi}^{(n)}(A)\neq 0)\leq \mathbb{P}(\widetilde{\chi}^{(n,2)}(A_1)>0)\leq c^4_0/(8n)\to 0,
\end{align*}which completes the proof.\end{proof}

\section{Integral inequalities of two-component log-gas}\label{twotwo}
In this section, we will prove several useful inequalities regarding the two-component log-gas, which is one of the crucial steps
in proving the convergence of the factorial moments of $\widetilde \chi^{(n)}$ (see Lemma \ref{lem8}).

Let  $A=(0,c_0),\ n>2k,$ by the definition of $\rho^{(k,n)}$, we have\begin{equation}\label{times}\begin{split}\mathbb{E}\rho^{(k,n)}(A^k)&=\frac{n!}{(n-2k)!2^kG_n}\int_{\Sigma_{n,k,c_0/n}}|J_n(\lambda_1,\cdots, \lambda_n)| d\lambda_1\cdots d\lambda_n,\end{split}
\end{equation}where $J_n$ is defined in \eqref{46} and \begin{align*}\numberthis \label{sigma}\begin{split}\Sigma_{n,k,c}=\{&(\lambda_1,\cdots, \lambda_n)\in\mathbb{R}^n: |\lambda_j-\lambda_{j-k}|<c,\forall\ n-k< j\leq n\},\end{split}\end{align*}
i.e.,  $\Sigma_{n,k,c}$  is the set $(\lambda_1,\cdots, \lambda_n)\in\mathbb{R}^n$ with $k$ pairs $(\lambda_j, \lambda_{j-k})$ such that  $|\lambda_j-\lambda_{j-k}|<c$.

We will first prove the  inequality \eqref{21} below regarding the two-component log-gas. The significance of such type inequality is that it will imply the bounds between the integration of the joint density over the set  $\Sigma_{n,k,c_0/n}$, i.e., $\mathbb{E}\rho^{(k,n)}(A^k)$ and the partition function $G_{n-2k, k}$ of the two-component log-gas which consists of $n-2k$ particles with charge $q=1$ and $k$ particles with charge $q=2$
 (see Lemma \ref{lem12}).


For $0\leq l\leq k$,  let's denote the following integral of the two-component log-gas\begin{align*}&{E_{n,k,l}}(c):=\int_{\Sigma_{n-l,k-l,c}}d\lambda_1\cdots d\lambda_{n-l}e^{-\sum\limits_{i=1}^{n-l}q_i\lambda_i^2/2}
\prod_{j<m}|\lambda_j-\lambda_m|^{q_jq_m}\Big|_{q_s=1+\chi_{\{s\leq l\}}},
\end{align*}where $\Sigma_{n-l,k-l,c}$ is    defined via \eqref{sigma}.
By definition of $G_{n_1,n_2}$ (recall \eqref{gaa}), we first have $${E_{n,k,k}}(c)=G_{n-2k,k}.
$$
We also have\begin{align}\label{exd}{{E_{n,k,0}}(c)}{ }=\int_{\Sigma_{n,k,c}}|J_n(\lambda_1,\cdots, \lambda_n)| d\lambda_1\cdots d\lambda_n,
\end{align}which implies 
\begin{align}\label{exde}\mathbb{E}\rho^{(k,n)}(A^k)&=\frac{n!}{(n-2k)!2^kG_n}{{E_{n,k,0}}(c_0/n)}{ }.
\end{align}
 We will  show that (for $0< 2n c^2<1$)\begin{equation}\label{21}(1-n c^2){c^{2}}\leq\frac{{E_{n,k,l-1}}(c)}{{E_{n,k,l}}(c)}\leq{c^{2}}.
\end{equation}In fact, after changing the order of variables, we can rewrite\begin{align*}&{E_{n,k,l-1}}(c)=\int_{\Sigma_{n-l-1,k-l,c}}d\lambda_1\cdots d\lambda_{n-l-1}e^{-\sum\limits_{i=1}^{n-l-1}q_i\lambda_i^2/2}
\prod_{1\leq j<m\leq n-l-1}|\lambda_j-\lambda_m|^{q_jq_m}\\ &\times\int_{\mathbb{R}}d x_1\int_{x_1-c}^{x_1+c}d x_2|x_1-x_2|e^{-x_1^2/2-x_2^2/2}\prod_{j=1}^2\prod_{m=1}^{n-l-1}|x_j-\lambda_m|^{q_m}\Big|_{q_s=1+\chi_{\{s\leq l-1\}}},
\end{align*}and \begin{align*}&{E_{n,k,l}}(c)=\int_{\Sigma_{n-l-1,k-l,c}}d\lambda_1\cdots d\lambda_{n-l-1}e^{-\sum\limits_{i=1}^{n-l-1}q_i\lambda_i^2/2}
\prod_{1\leq j<m\leq n-l-1}|\lambda_j-\lambda_m|^{q_jq_m}\\ &\times\int_{\mathbb{R}}d x_1e^{-x_1^2}\prod_{m=1}^{n-l-1}|x_1-\lambda_m|^{2q_m}\Big|_{q_s=1+\chi_{\{s\leq l-1\}}}.
\end{align*}Then \eqref{21} follows from \eqref{slt}  by taking  \begin{equation}\label{fsx}F(x)=e^{-x^2/2}\prod_{j=1}^{n-l-1}|x-\lambda_m|^{q_m}.\end{equation} By \eqref{21} we will have\begin{align}\label{22}&{{E_{n,k,l}}(c)}\leq\left(c^2\right)^{k-l}{E_{n,k,k}}(c)
=c^{2(k-l)}G_{n-2k,k}.
\end{align}
For $n>2k,$ given any interval $A$, let's denote \begin{align}\label{dsgs}\Sigma_{n,k,A}=\{&(\lambda_1,\cdots, \lambda_n)\in\mathbb{R}^n: |\lambda_j-\lambda_{j-k}|\in A,\forall\ n-k< j\leq n\}.
\end{align}For $0\leq l\leq k$, let's denote
$$\label{dddss}{E_{n,k,l}}(A):=\int_{\Sigma_{n-l,k-l,A}}d\lambda_1\cdots d\lambda_{n-l}e^{-\sum\limits_{i=1}^{n-l}q_i\lambda_i^2/2}
\prod_{j<m}|\lambda_j-\lambda_m|^{q_jq_m}$$where ${q_s=1+\chi_{\{0<s\leq l\}}}$ and  $\Sigma_{n-l,k-l,A}$ is defined via \eqref{dsgs}. Then we have\begin{equation}\label{dldl}{{E_{n,k,0}}(A)}{}=\int_{\Sigma_{n,k,A}}|J_n(\lambda_1,\cdots, \lambda_n)|  d\lambda_1\cdots d\lambda_n
\end{equation}and \begin{equation}\label{ekn}{E_{n,k,k}}(A)=G_{n-2k,k}.\end{equation} 
With such notations, as before, we have 
 \begin{equation}\label{ea}\begin{split}\mathbb{E}\rho^{(k,n)}(A^k)=&\frac{n!}{(n-2k)!2^kG_n}\int_{\Sigma_{n,k,A/n}}|J_n(\lambda_1,\cdots, \lambda_n)| d\lambda_1\cdots d\lambda_n\\ =&\frac{n!}{(n-2k)!2^k}\frac{E_{n,k,0}(A/n)}{G_{n}}
 ,\end{split}
\end{equation}
We  also need inequalities similar to \eqref{21}.\begin{lem}\label{lem12}If $ A\subset(0,c_1)$,
$2n c_1^2\in(0,1),$ $n>2k,$ $n,k$ are positive integers, then we have \begin{align*}&(1-nc_1^2)^k\left(2\int_Audu\right)^kG_{n-2k,k}\leq{E_{n,k,0}}(A)\leq\left(2\int_Audu\right)^kG_{n-2k,k}.
\end{align*}\end{lem}\begin{proof}Let $A_1=A\cup(-A),$ after changing the order of variables, we can rewrite\begin{align*}&{E_{n,k,l-1}}(A)=\int_{\Sigma_{n-l-1,k-l,A}}\!\!\!\!d\lambda_1\cdots d\lambda_{n-l-1}e^{-\sum\limits_{i=1}^{n-l}q_i\lambda_i^2/2}
\!\!\prod_{1\leq j<m\leq n-l-1}\!\!|\lambda_j-\lambda_m|^{q_jq_m}\\ &\times\int_{\mathbb{R}}d x_1\int_{x_1+A_1}d x_2|x_1-x_2|e^{-x_1^2/2-x_2^2/2}\prod_{j=1}^2\prod_{m=1}^{n-l-1}|x_j-\lambda_m|^{q_m}\Big|_{q_s=1+\chi_{\{s\leq l-1\}}},
\end{align*}and \begin{align*}&{E_{n,k,l}}(A)=\int_{\Sigma_{n-l-1,k-l,A}}d\lambda_1\cdots d\lambda_{n-l-1}
\prod_{1\leq j<m\leq n-l-1}|\lambda_j-\lambda_m|^{q_jq_m}\\ &\times\int_{\mathbb{R}}d x_1e^{-x_1^2}\prod_{m=1}^{n-l-1}|x_1-\lambda_m|^{2q_m}\Big|_{q_s=1+\chi_{\{s\leq l-1\}}}.
\end{align*}Taking $F(x)$ as in \eqref{fsx} again, by \eqref{sdsss} we have\begin{align*}&(1-nc_1^2)\cdot2\int_Audu\leq\frac{{E_{n,k,l-1}}(A)}{{E_{n,k,l}}(A)}\leq2\int_Audu,
\end{align*}and the result follows by induction and \eqref{ekn}.
\end{proof}

\section{Proof of  Theorem  \ref{thm1}}
By Lemma \ref{lem4} and the moment method,  Theorem \ref{thm1} will be proved if we can prove 
the following convergence of the factorial moment\begin{align}\label{20}&\lim_{n\to+\infty}\mathbb{E}\left(\frac{(\widetilde{\chi}^{(n)}(A))!}{(\widetilde{\chi}^{(n)}(A)-k)!}
\right)=\left(\frac{1}{4}\int_Audu\right)^k
\end{align} for every positive integer $k$ and bounded interval $ A\subset\mathbb{R}_{+}$.
 Actually, combining Lemma \ref{lem9},  \eqref{20} is equivalent to  \begin{lem}\label{lem8}For any bounded interval $A\subset\mathbb{R}_+$ and
any positive integer $k\geq 1$,   we have\begin{align*}&
\mathbb{E}\left(\frac{(\widetilde{\chi}^{(n)}(A))!}{(\widetilde{\chi}^{(n)}(A)-k)!}\right)
-\left(\int_Audu\right)^k\frac{G_{n-2k,k}}{G_{n}}\to 0
\end{align*}as $n\to+\infty$.\end{lem}
 We will first use    Lemma \ref{lem14} to prove that\begin{align}\label{25}&\lim_{n\to+\infty}\left(\mathbb{E}\frac{(\widetilde{\chi}^{(n)}(A))!}{(\widetilde{\chi}^{(n)}(A)-k)!}
-\mathbb{E}\rho^{(k,n)}(A^k)\right)=0,
\end{align}and then use Lemma \ref{lem12} to prove that\begin{align}\label{26}&
\lim_{n\to+\infty}\left(\mathbb{E}(\rho^{(k,n)}(A^k))
-\left(\int_Audu\right)^k\frac{G_{n-2k,k}}{G_{n}}\right)=0,
\end{align}then Lemma \ref{lem8} follows from \eqref{25} and \eqref{26}, and hence we complete  the proof of Theorem \ref{thm1}.

For the rest of the article, 
  for any bounded interval  $ A\subset\mathbb{R}_{+}$, let $c_1$ be such that $A\subset(0,c_1)$, and $A_1=(0,2c_1)$, then $A\subset A_1$;  let's denote $c_n=c_1/n$, then  $ 2n (2c_n)^2=8n^{-1} c^2_1\in(0,1)$ for $n$ large enough.     By  \eqref{exde}, \eqref{22} with $l=0$ and Lemma \ref{lem9},  we have \begin{align}&\label{27}\mathbb{E}\rho^{(k,n)}(A_1^k)
=\frac{n!}{(n-2k)!2^k}\frac{E_{n,k,0}(2c_n)}{G_{n}}
\\ \nonumber\leq&\frac{n!}{(n-2k)!2^k}\frac{G_{n-2k,k}}{G_{n}}\left(2c_n\right)^{2k}\leq \frac{n^{2k}}{2^k}
2^{-2k}\left(\frac{2c_1}{n}\right)^{2k}=2^{-k}c_1^{2k}.
\end{align}Let $a$ be defined in Lemma \ref{lem14}, then we have\begin{align*}&\rho^{(k,n)}(A_1^k)\geq\frac{(a+1)!}{(a+1-2k)!2^k}\geq\frac{(a+1-2k)_+^{2k}}{2^k},
\end{align*} and hence \begin{align*}&\mathbb{E}(a+1-2k)_+^{2k}\leq 2^k\mathbb{E}\rho^{(k,n)}(A_1^k)\leq c_1^{2k},
\end{align*}here we denote $f_+:=\max(f,0).$ Since $a,k\in\mathbb{Z},\ a\geq0, \ k\geq 1$, by H\"older's inequality we have\begin{align*}&\mathbb{E}(a+1-2k)_+^{k}\leq (\mathbb{E}(a+1-2k)_+^{2k})^{\frac{1}{2}} (\mathbb{P}(a\geq 2))^{\frac{1}{2}}\leq c_1^{k} (\mathbb{P}(a\geq 2))^{\frac{1}{2}}.
\end{align*}Moreover, it's easy to check \begin{align*}&(a-1)_+\leq \max\left(2(a+1-2k)_+,(4k-4){\chi}_{a\geq 2}\right), \end{align*}and thus \begin{align*}(a-1)_+^k\leq 2^k(a+1-2k)_+^k+(4k-4)^k\chi_{a\geq 2},\end{align*} hence, we have
\begin{align*}\mathbb{E}(a-1)_+^k&\leq 2^k\mathbb{E}(a+1-2k)_+^k+(4k-4)^k\mathbb{P}(a\geq 2)\\
&\leq 2^kc_1^{k} (\mathbb{P}(a\geq 2))^{\frac{1}{2}}+(4k-4)^k\mathbb{P}(a\geq 2).
\end{align*} 
On the other hand, $a\geq 2$ is equivalent to $\widetilde{\chi}^{(n,2)}(A_1)>0, $ by Lemma \ref{lem11} we have\begin{align*}&\mathbb{P}(a\geq 2)= \mathbb{P}(\widetilde{\chi}^{(n,2)}(A_1)>0)\leq 2c^4_1/n\to 0,
\end{align*}and thus we further have \begin{align}\label{28}&\lim_{n\to+\infty}\mathbb{E}(a-1)_+^k=0.
\end{align}  By \eqref{add} in Lemma \ref{lem14} we have $$(\widetilde{\chi}^{(n)}(A))^k\leq2\rho^{(k,n)}(A^k)\,\,\mbox{or}\,\,(\widetilde{\chi}^{(n)}(A))^k\leq2k(k-1)a(\widetilde{\chi}^{(n)}(A))^{k-1},
$$ therefore,   \begin{align*}&(\widetilde{\chi}^{(n)}(A))^k\leq\max(2\rho^{(k,n)}(A^k),(2k(k-1)a)^k),
\end{align*}and thus we have \begin{align*}&\mathbb{E}(\widetilde{\chi}^{(n)}(A))^k\leq2\mathbb{E}(\rho^{(k,n)}(A^k))+(2k(k-1))^k\mathbb{E}(a^k).
\end{align*}By  \eqref{27},  \eqref{28} and the fact that $\mathbb E\rho^{(k,n)}(A^k)\leq \mathbb E\rho^{(k,n)}(A_1^k)$ since $A\subset A_1$,  we further have  \begin{align}\label{29}&\limsup_{n\to+\infty}\mathbb{E}(\widetilde{\chi}^{(n)}(A))^k<+\infty.
\end{align} Note that \eqref{25} is clearly true for $k=1$ by definitions. For $k\geq 2$, by \eqref{aaad} in Lemma \ref{lem14}, H\"older's inequality, \eqref{28} and \eqref{29}, we have \begin{align*}0&\leq\mathbb{E}\left(\frac{(\widetilde{\chi}^{(n)}(A))!}{(\widetilde{\chi}^{(n)}(A)-k)!}-\rho^{(k,n)}(A^k)\right)
\\&\leq k(k-1)\mathbb{E}((a-1)_+(\widetilde{\chi}^{(n)}(A))^{k-1})\\&\leq k(k-1)(\mathbb{E}((a-1)_+^k))^{1/k}(\mathbb{E}(\widetilde{\chi}^{(n)}(A)^{k}))^{1-1/k}\to0
\end{align*}as $n\to+\infty,$ which finishes the proof of \eqref{25}.

Now we prove \eqref{26}. By \eqref{ea} and changing of variables, we have\begin{align*}&
\mathbb{E}(\rho^{(k,n)}(A))
-\left(\int_Audu\right)^k\frac{G_{n-2k,k}}{G_{n}}\\=&
\frac{n!}{(n-2k)!2^k}\frac{{E_{n,k,0}}(A/n)}{G_{n}}-\left(\int_{A/n}udu\right)^k
\frac{n^{2k}G_{n-2k,k}}{G_{n}}\\=&
\frac{n^{2k}}{2^kG_{n}}\left({E_{n,k,0}}(A/n)-\left(2\int_{A/n}udu\right)^k
G_{n-2k,k}\right)\\&-\left(n^{2k}-\frac{n!}{(n-2k)!}\right)\frac{{E_{n,k,0}}(A/n)}{2^kG_{n}}.
\end{align*}We first notice that\begin{align*}
0&\leq n^{2k}-\frac{n!}{(n-2k)!}=n^{2k}-\prod_{j=0}^{2k-1}(n-j)=n^{2k}-n^{2k}\prod_{j=0}^{2k-1}(1-j/n)\\&\leq n^{2k}-n^{2k}\left(1-\sum_{j=0}^{2k-1}j/n\right)=n^{2k}\sum_{j=0}^{2k-1}j/n=n^{2k-1}k(2k-1). 
\end{align*}We also have $A/n\subset(0,c_1/n)$ and  $2n(c_1/n)^2\in(0,1)$ for $n$ large enough, then by \eqref{exd}, \eqref{22} and \eqref{dldl}, we have \begin{align*}
0&\leq {E_{n,k,0}}(A/n)\leq {E_{n,k,0}}(c_1/n)\leq G_{n-2k,k}(c_1/n)^{2k}.
\end{align*} Therefore,  using Lemma \ref{lem9} we have \begin{align*}
0&\leq \left(n^{2k}-\frac{n!}{(n-2k)!}\right)\frac{{E_{n,k,0}}(A/n)}{2^kG_{n}}\\&\leq n^{2k-1}k(2k-1) \frac{G_{n-2k,k}}{2^kG_{n}}(c_1/n)^{2k}\\&= n^{-1}k(2k-1)2^{-3k}c_1^{2k}.
\end{align*}By Lemma \ref{lem9} and Lemma \ref{lem12}, we have\begin{align*}&
\frac{n^{2k}}{2^kG_{n}}\left|{E_{n,k,0}}(A/n)-\left(2\int_{A/n}udu\right)^k
G_{n-2k,k}\right|\\ \leq&\frac{n^{2k}}{2^kG_{n}}(1-(1-n(c_1/n)^2)^k)\left(2\int_{A/n}udu\right)^k
G_{n-2k,k}\\ \leq&\frac{n^{2k}}{2^kG_{n}}(kn(c_1/n)^2)\left(2\int_{0}^{c_1/n}udu\right)^k
G_{n-2k,k}\\ =&\frac{n^{2k}}{2^kG_{n}}(kc_1^2/n)\left(c_1/n\right)^{2k}
G_{n-2k,k}\\=&\frac{G_{n-2k,k}}{2^kG_{n}}(kc_1^{2k+2}/n)=\frac{kc_1^{2k+2}}{2^{3k}{n}}.
\end{align*}Therefore, we have \begin{align*}&
\left|\mathbb{E}(\rho^{(k,n)}(A))
-\left(\int_Audu\right)^k\frac{G_{n-2k,k}}{G_{n}}\right| \leq\frac{kc_1^{2k+2}+k(2k-1)c_1^{2k}}{2^{3k}{n}},
\end{align*}which implies \eqref{26}. Therefore, we finish the proof of Lemma \ref{lem8} and thus the whole proof of Theorem \ref{thm1}.

\end{document}